\documentclass[final,1p,times]{elsarticle}

\usepackage{amsmath, epsfig, cite}
\usepackage{amssymb}
\usepackage{amsfonts}
\usepackage{latexsym}
\usepackage{graphicx}
\usepackage{amsthm,array}
\usepackage{color,xcolor}
\usepackage{algorithm}
\usepackage{algorithmic,float}
\usepackage{setspace}
\usepackage{array}

\newtheorem{theorem}{Theorem}[section]

\newtheorem{example}[theorem]{Example}

\DeclareMathOperator{\Spr}{Spr}



\begin{document}

\begin{frontmatter}



\title{Polynomial Solutions to the First Order Difference Equations in the Bivariate Difference Field}

\author{Yarong Wei}\ead{yarongwei@email.tjut.edu.cn}
\address{School of Science, Tianjin University of Technology, Tianjin 300384, PR China}

\begin{abstract}
The bivariate difference filed $(\mathbb{F}(\alpha, \beta), \sigma)$ provides an algebraic framework for a sequence satisfying a recurrence of order two and it could transform the summation involving a sequence satisfying a recurrence of order two into the first order difference equations in the bivariate difference field. Based on it, we present an algorithm for finding all the polynomial solutions of such equations in the bivariate difference field, and show an upper bound on the degree for polynomial solutions which is sufficient to compute polynomial solution by using the undetermined method.
\end{abstract}




\end{frontmatter}


\section{Introduction}
Karr developed an algorithm for indefinite summation \citep{Karr-1981,Karr-1985} based on the theory of difference field. He introduced so called $\Pi\Sigma-$fields, in which a sequence satisfying a recurrence of order one can be easily described.
This consideration could deal with not only hypergeometric terms, like Gosper's algorithm \citep{Gosper-1978}, and $q$-hypergeometric terms, like \citep{Paule-1997}, but terms including the harmonic numbers. There were several extensions of Karr's algorithm. Bronstein \citep{Bronstein-2000} described how to compute the solutions of generalized difference equations from the point of view of differential fields. Abramov, Bronstein, Petkov\v{s}ek and Schneider \citep{Abramov-2021} gave a complete algorithm to compute all hypergeometric solutions of homogeneous linear difference equations and rational solutions of parameterized linear difference equations in the $\Pi\Sigma^*$-fields.

Recently, Hou and Wei \citep{Hou-Wei-2023} introduced so called bivariate difference field, in dealing with the summation involving a sequence satisfying recurrence of order two. They transformed the summation into the first order difference equations in the bivariate difference field, and gave an algorithm to find a rational solution of the bivariate difference field. Base on this, we want to develop an algorithm to find all the polynomial solutions of such equations, i.e., find all the polynomials $g((\alpha, \beta))$ of the first order difference equation
\begin{equation} \label{eq-dif}
 a(\alpha, \beta)\sigma(g(\alpha, \beta))+b(\alpha, \beta)g(\alpha, \beta)=f(\alpha, \beta),
\end{equation}
where $a(\alpha, \beta), b(\alpha, \beta), f(\alpha, \beta)$ are given non-zero rational functions in the bivariate difference filed and $\sigma$ is given isomorphism of the bivariate difference field.

Take, for example, the Fibonacci sequence $F_n$ which satisfies
\[ F_{n+2} = F_{n+1} + F_n, \quad F_0=0,\quad F_1=1. \]
According to the definition of bivariate difference field \citep{Hou-Wei-2023}, we regard the summations involving $F_n$ lie in the difference field $(\mathbb{R}(\alpha, \beta), \sigma)$ with
\[
\sigma|_{\mathbb{R}} = {\rm id}, \quad  \sigma(\alpha)=\beta,  \quad\text{and}\quad \sigma(\beta)=\alpha+\beta.
\]
For the summation $\sum_{n=0}^mF_{n+1}^2$, we transform the summation to solve the difference equation
 \[\sigma(g(\alpha, \beta))-g(\alpha, \beta)=\beta^2,\]
By the algorithm about finding all the polynomial solutions, one can easily find
\[\sigma(\alpha\beta)-\alpha\beta=\beta^2,\]
which implies
\[\sum_{n=0}^mF_{n+1}^2=F_{m+1}F_{m+2},\]
and it is a result of Lucas, published by Koshy \citep[Eq(5.5)]{Koshy-2001}.

The outline of the article is as follows. In Section \ref{sec-2}, we will give some properties and results in the bivariate difference filed which play important role to compute all the polynomial solutions of (\ref{eq-dif}).
Then we will present an algorithm to compute all the polynomial solutions, and show an upper bound on the degree for polynomial solutions which is sufficient to solve the first order difference equation by using the undetermined method in Section \ref{sec-3}, moreover, we show that how to use the algorithm to solve the summation involving a sequence satisfying recurrence of order two via some examples.

\section{Bivariate Difference Field}\label{sec-2}
In this section, we give some notations and basic results about the bivariate difference field $(\mathbb{F}(\alpha, \beta), \sigma)$.

First, we recall the definition of the bivariate difference equation which introduced in \citep{Hou-Wei-2023}.
Let $(\mathbb{F},\sigma)$ be a difference field, i.e., $\sigma$ is an automorphism of $\mathbb{F}$. The \emph{bivariate difference field} extension $(\mathbb{F}(\alpha, \beta), \sigma)$ of $(\mathbb{F},\sigma)$ is the field with
\begin{enumerate}
\item $\alpha,\beta$ being algebraically independent transcendental elements over $\mathbb{F}$.
\item $\sigma$ being an automorphism of $\mathbb{F}(\alpha,\beta)$ and
\begin{align}\label{eq-sigma}
\sigma(\alpha)=\beta, \ \ \ \sigma(\beta)=u\alpha+v\beta
\end{align}
where $v\in\mathbb{F}, u\in\mathbb{F}\setminus\{0\}$.
 \end{enumerate}

Note $\sigma$ is a isomorphism of the bivariate difference field $(\mathbb{F}(\alpha, \beta), \sigma)$ and $\sigma(\alpha,\beta)=(\alpha,\beta)A$, where $A=\left(
    \begin{array}{cc}
      0 & u \\
      1 & v \\
    \end{array}
\right)$ is a unique $2\times2$ matrix defined by $u$ and $v$, hence we could also describe the  bivariate difference field $(\mathbb{F}(\alpha, \beta), \sigma)$ by the bivariate difference field $\mathbb{F}(\alpha, \beta)$ with matrix $A$.

Denote $\mathbb{F}[\alpha, \beta]$ be the set of polynomials of $\alpha, \beta$. For a polynomial $f \in \mathbb{F}[\alpha, \beta]$, let $\deg f$ denote the \emph{total degree} in variables $\alpha, \beta$, and we  use the total degree lexicographical ordering so that $\alpha \succ \beta$. Moreover, let $\sum_{i=0}^{\deg f}f_i$ denote the \emph{homogeneous decomposition }of $f$, where $f=\sum_{i=0}^{\deg f}f_i$ and $f_i$ is a homogeneous polynomial with $\deg f_i=i$ for every $0\leq i\leq \deg f$.
An element $a\in(\mathbb{F}(\alpha, \beta), \sigma)$ is called \emph{semi-invariant} if $\sigma a=ua$ for $u\in\mathbb{F}(\alpha, \beta)\setminus\{0\}$, and we write $\mathbb{F}(\alpha,\beta)^{\sigma}$ for the semi-invariants of $\mathbb{F}(\alpha,\beta)$.

Now, we show a result about the structure of homogeneous polynomials in the set of semi-invariants, which is given in \citep{Hou-Wei-2023}.

\begin{theorem}[\citep{Hou-Wei-2023}]\label{thm-2-2}
In the bivariate difference field $(\mathbb{F}(\alpha, \beta), \sigma)$ with $\sigma|_\mathbb{F}={\rm id}$ and the matrix
\begin{align*}
  A =
  \left(
    \begin{array}{cc}
      0 & u \\
      1 & v \\
    \end{array}
  \right)
  \end{align*}
with the eigenvectors $X_1,X_2$ corresponding to the eigenvalues $\lambda_1,\lambda_2$, respectively. If $\lambda_1\neq\lambda_2\in\mathbb{F}$ and $\lambda_1/\lambda_2$ is not a root of unit, then for the homogeneous polynomial $p\in\mathbb{F}(\alpha,\beta)^{\sigma}$ with $\deg p=m$, there exist $c\in\mathbb{F}$ and $0\leq i\leq m$ such that
\begin{align}\label{eq-2-1}
p=c\cdot h_1^{i}\cdot h_2^{m-i},
\end{align}
where $h_1(\alpha,\beta)=(\alpha,\beta)X_1$, $h_2(\alpha,\beta)=(\alpha,\beta)X_2$
\end{theorem}

Basic on the this, we consider the structure of rational functions in the set of semi-invariants.

\begin{theorem}\label{thm-2-3}
In the bivariate difference field $(\mathbb{F}(\alpha, \beta), \sigma)$ with the same condition of Theorem \ref{thm-2-2}, let $g=\frac{p}{q}\in\mathbb{F}(\alpha,\beta)^{\sigma}$ be a rational function, where $p,q\in\mathbb{F}[\alpha,\beta]$ are two homogenous polynomials, then there exists $i$ with $-\deg q\leq i \leq\deg p$ such that
\[\frac{\sigma g}{g}=\lambda_1^{i}\lambda_{2}^{\deg p-\deg q-i}.\]
\end{theorem}

\begin{proof}
Let $t=\gcd(p,q)$, then $g=\frac{p/t}{q/t}$, where $p/t, q/t\in\mathbb{F}[\alpha,\beta]$ are primely homogenous polynomials. Since $g\in\mathbb{F}(\alpha,\beta)^{\sigma}$, there exists $k\in\mathbb{F}$ such that
\begin{align*}
\frac{\sigma p}{\sigma q}&=k\frac{p}{q},\\
\frac{\sigma(p/t)}{\sigma (q/t)}&=k\frac{p/t}{q/t},\\
\sigma(p/t)\cdot (q/t)&= k\cdot \sigma(q/t)\cdot (p/t),
\end{align*}
moreover
\[(p/t)\mid\sigma(p/t),\quad (q/t)\mid \sigma(q/t),\]
which implies $p/t, q/t\in\mathbb{F}(\alpha,\beta)^{\sigma}$. Then by \eqref{eq-2-1}, there exist $c_p,c_q
in\mathbb{F}$, $0\leq i\leq \deg (p/t)$ and $0\leq j\leq\deg(q/t)$ such that
\[
\frac{\sigma(p/t)}{p/t}=\frac{\sigma\left(c_p\cdot h_1^{i}\cdot h_2^{\deg(p/t)-i}\right)}{c_p\cdot h_1^{i}\cdot h_2^{\deg(p/t)-i}}=\frac{c_p\cdot\lambda_1^i\lambda_2^{\deg(p/t)-i}\cdot h_1^{i}\cdot h_2^{\deg(p/t)-i}}{c_p\cdot h_1^{i}\cdot h_2^{\deg(p/t)-i}}=\lambda_1^i\lambda_2^{\deg(p/t)-i},
\]
\[
\quad\frac{\sigma(q/t)}{q/t}=\frac{\sigma\left(c_q\cdot h_1^{i}\cdot h_2^{\deg(q/t)-i}\right)}{c_q\cdot h_1^{i}\cdot h_2^{\deg(q/t)-i}}=\frac{c_q\cdot\lambda_1^j\lambda_2^{\deg(q/t)-j}\cdot h_1^{i}\cdot h_2^{\deg(q/t)-i}}{c_q\cdot h_1^{i}\cdot h_2^{\deg(q/t)-i}}=\lambda_1^j\lambda_2^{\deg(q/t)-j},
\]
and
\[\frac{\sigma g}{g}=\frac{p/t}{q/t}\cdot\frac{q/t}{p/t}
=\frac{\lambda_1^i\lambda_2^{\deg(p/t)-i}}{\lambda_1^j\lambda_2^{\deg(q/t)-j}}=\lambda_1^{i-j}\lambda_2^{\deg p-\deg q-(i-j)}.\]
Obviously, $-\deg q\leq i-j\leq\deg p$, which implies the result.
\end{proof}

Next, we show two results about the spread set. Recall that for any two polynomials $p,q\in\mathbb{F}[\alpha,\beta]\backslash\{0\}$, $\Spr_{\sigma}(p,q)$ denotes the \emph{spread} of $p$ and $q$ with respect to $\sigma$, i.e.,
\[
\Spr_{\sigma}(p,q)=\{m|m\in \mathbb{N}\ \ \text{such that}\  \deg\left(\gcd(p,\sigma^{m}q)\right)>0\}.
\]

\begin{theorem}\label{thm-prime}
Let $(\mathbb{F}(\alpha,\beta), \sigma)$ be a bivariate difference field with the same condition of Theorem \ref{thm-2-2}, and $a,b\in \mathbb{F}[\alpha,\beta]$ be two primely homogenous polynomials. Then $\Spr_{\sigma}(a,b)$ is a finite set.
\end{theorem}

\begin{proof}
Let $a_i,b_j$ be any two irreducible factors of $a,b$, respectively. Obviously, $a_i,b_j\in\mathbb{F}[\alpha,\beta]$ are two primely homogenous polynomials, it suffices to show $\Spr_{\sigma}(a_i,b_j)$ is a finite set.

First, let $h_1,h_2$ are two polynomials which defined in Theorem \ref{thm-2-2}. By the definition of $h_1,h_2$, we have $\gcd(h_1,h_2)=1$, and for any $m\in\mathbb{Z}$, $\sigma^{m}h_2=\lambda_2^mh_2$, then $\Spr_{\sigma}(h_1,h_2)=\Spr_{\sigma}(h_2,h_1)=\emptyset$.

Next, we consider the result from three cases. If $a_i, b_j\in\mathbb{F}[\alpha,\beta]^{\sigma}$, then by Theorem \ref{thm-2-2}, there exists $c_i, c_j\in\mathbb{F}$, $0\leq i'\leq \deg a_i, 0\leq j'\leq\deg b_j$ such that
\[a_i=c_i\cdot h_1^{i'}\cdot h_2^{\deg a_i-i'}, b_j=c_j\cdot h_1^{j'}\cdot h_2^{\deg b_j-j'}.\]
Since $\gcd(a_i,b_j)=1$ and $\Spr_{\sigma}(h_1,h_2)=\Spr_{\sigma}(h_2,h_1)=\emptyset$, the result holds for this case.

If $a_i\in\mathbb{F}[\alpha,\beta]^{\sigma}$ and $b_j\notin\mathbb{F}[\alpha,\beta]^{\sigma}$, then by Theorem \ref{thm-2-2} and $b_j$ is irreducible, there exists $c_i\in\mathbb{F}$, $0\leq i'\leq \deg a_i$ such that $a_i=c_i\cdot h_1^{i'}\cdot h_2^{\deg a_i-i'}$, and $h_t\nmid b_j$ for $t=1,2$. Since $\sigma^{m}h_t=\lambda_t^mh_t$ for any $m\in\mathbb{Z}$ and $t=1,2$, the result holds for this case. Moreover, it is similar to get the result for $a_i\notin\mathbb{F}[\alpha,\beta]^{\sigma}$ and $b_j\in\mathbb{F}[\alpha,\beta]^{\sigma}$.

If $a_i,b_j\notin\mathbb{F}[\alpha,\beta]^{\sigma}$, then $\Spr_{\sigma}(a_i,a_i)=\Spr_{\sigma}(b_j,b_j)=\emptyset$. Suppose $m_1<m_2\in\Spr_{\sigma}(a_i,b_j)$, then $a_i\mid\sigma^{m_1}(b_j)$ and $a_i\mid\sigma^{m_2}(b_j)$, moreover, since $a_i,b_j$ are irreducible, $a_i\mid\sigma^{m_2-m_1}(a_j)$, which is contradict with  $\Spr_{\sigma}(a_i,a_i)=\emptyset$.

\end{proof}

%
%
%

\begin{theorem}\label{thm-2-1}
Let $(\mathbb{F}(\alpha, \beta), \sigma)$ be a bivariate difference field, and $g(\alpha,\beta)\in\mathbb{F}(\alpha,\beta)$ be a rational function. If
\[\frac{\sigma g}{g}=\frac{p}{q},\]
where $p,q\in\mathbb{F}[\alpha, \beta]$ are polynomials with $\gcd_{\sigma}(p,q)=1$, then
\[\text{either }p, q\in\mathbb{F}, \text{ or } \Spr_{\sigma}(p,q)\bigcup\Spr_{\sigma}(p,q)\neq\emptyset.\]
\end{theorem}

\begin{proof}
First, we consider the case $g(\alpha,\beta)\in\mathbb{F}[\alpha,\beta]$. Let
\[g=v\prod_{i\in S}h_i^{e_{i}},\]
be the irreducible factor decomposition of $g$ with respect to $\sigma$, where $v\in\mathbb{F}$, $e_{i}\in\mathbb{N^+}$, $S\subseteq\mathbb{N}$ and $h_i\in\mathbb{F}[\alpha,\beta]$ are monic irreducible polynomials.
Assume $t=\gcd(g,\sigma g)$, then there exist four sets $T_1,T_2,P,Q\in S$ such that
\[
t=v_1\prod_{i\in T_1}h_i^{n_i}=v_2\prod_{i\in T_2}\sigma(h_i)^{m_i},
\]
and
\[
p=\frac{\sigma g}{t}=u_1\prod_{i\in P}\sigma(h_i)^{l_i},\quad q=\frac{g}{t}=u_2\prod_{i\in Q}h_i^{k_i},
\]
where $v_1,v_2,u_1,u_2\in\mathbb{F}\setminus{0}, 0<n_i,m_i,l_i.k_i<e_i$, i.e.,
\[\frac{\sigma g}{g}=\frac{p\cdot t}{q\cdot t}=\frac{u_1v_2}{u_2v_1}\frac{\prod_{i\in P}\sigma(h_i)^{l_i}\cdot\prod_{i\in T_2}\sigma(h_i)^{m_i}}{\prod_{i\in Q}h_i^{k_i}\cdot\prod_{i\in T_1}h_i^{n_i}}.\]
Note that, since the isomorphism $\sigma$ keeps the total degree of polynomials unchanged and $h_i$ are irreducible polynomials, $\sum l_i=\sum k_i$, and $\sum m_i=\sum n_i$.

Suppose $P\bigcap Q=\emptyset$, otherwise $1\in\Spr_{\sigma}(p,q)$.
If $P=\emptyset$, then $Q=\emptyset$ and $p,q\in\mathbb{F}$, since $\deg g=\deg(\sigma g)$. Then we only need consider the case $P,Q\neq\emptyset$.

Assume $i_0\in P$, by $P\bigcap Q=\emptyset$, we have $i_0\notin Q$ and $i_0\in T_1$, then there exists $i_1\in T_2$ such that
\begin{align}\label{eq-2-2}
\frac{\sigma(h_{i_1})}{h_{i_0}}\in\mathbb{F},
\end{align}
note it is possible that $i_0=i_1$. If $i_1\in Q$, then by \eqref{eq-2-2}, we have
\begin{align*}
\frac{\sigma^2(h_{i_1})}{\sigma(h_{i_0})}\in\mathbb{F},
\end{align*}
which implies $2\in\Spr_{\sigma}(\sigma(h_{i_0}),h_{i_1})\in\Spr_{\sigma}(p,q)$;
If $i_1\notin Q$, then $i_1\in T_1$, and there exists $i_2\in T_2$ such that
\begin{align}\label{eq-2-3}
\frac{\sigma(h_{i_2})}{h_{i_1}}\in\mathbb{F},
\end{align}
note if $i_1=i_0$, then $h_{i_1}^2\mid t$, which implies the existence of $i_2$. If $i_2\in Q$, then by \eqref{eq-2-2} and \eqref{eq-2-3}, we have
\begin{align*}
\frac{\sigma^3(h_{i_2})}{\sigma(h_{i_0})}=\frac{\sigma^3(h_{i_2})}{\sigma^2(h_{i_1})}\cdot\frac{\sigma^2(h_{i_1})}{\sigma(h_{i_0})}\in\mathbb{F},
\end{align*}
which implies $3\in\Spr_{\sigma}(\sigma(h_{i_0}),h_{i_2})\in\Spr_{\sigma}(p,q)$;
If $i_2\notin Q$, then, since $P\bigcap Q=\emptyset$, $P\bigcup T_2=Q\bigcup T_1$, $\sum l_i=\sum k_i$, and $\sum m_i=\sum n_i$, there must exist $i_k\in Q$ such that
\[
\frac{\sigma^{k+1}(h_{i_k})}{h_{i_0}}\in\mathbb{F},
\]
which implies $k+1\in\Spr_{\sigma}(\sigma(h_{i_0}),h_{i_k})\in\Spr_{\sigma}(p,q)$, and set $\Spr_{\sigma}(p,q)$ is not a empty set.

Similarly, the result holds for the case $g(\alpha,\beta)\in\mathbb{F}(\alpha,\beta)$.

\end{proof}

\section{Polynomial Solutions}\label{sec-3}
In this section, for the bivariate field extension $(\mathbb{F}(\alpha,\beta),\sigma)$ with $\sigma|_{\mathbb{F}}=id$ and matrix $A =
  \left(
    \begin{array}{cc}
      0 & u \\
      1 & v \\
    \end{array}
  \right)$, we consider all of the polynomial solutions of difference equation
\begin{align}\label{eq-3-1}
a\sigma g+b g=f,
\end{align}
where $a,b,f\in\mathbb{F}[\alpha,\beta]$ are given polynomials.

First, we show a relationship between all of the polynomial solutions of equation \eqref{eq-3-1} and all of the polynomial solutions of equation
\begin{align}\label{eq-3-2}
a\sigma g+b g=0,
\end{align}
where $a,b\in\mathbb{F}[\alpha,\beta]$ are given in \eqref{eq-3-1}.

\begin{theorem}
Let $g^*$ be a polynomial solution of equation \eqref{eq-3-1}. Then
\begin{align*}
\{g\mid g \text{ is a polynomial solution of} &\text{ equation of \eqref{eq-3-1}}\}\\
&=\{g'+g^*\mid g' \text{ is a polynomial solution of equation of \eqref{eq-3-2}}\}.
\end{align*}
Moreover, there must exist a polynomial solution $g^*\in\mathbb{F}[\alpha,\beta]$ with $\deg g^*=\deg f-\max\{\deg a, \deg b\}$ of equation \eqref{eq-3-1}.
\end{theorem}

\begin{proof}
Let $g'$ be a polynomial solution of $\eqref{eq-3-2}$, then
\begin{align*}
a\sigma(g'+g^*)+b(g'+g^*)&=(a\sigma g'+bg')-(a\sigma g^*+bg^*)\\
&=f+0=f,
\end{align*}
which implies $g'+g^*$ is a polynomial solution of \eqref{eq-3-1}.

Let $g$ be a polynomial solution of $\eqref{eq-3-1}$, then
\begin{align*}
a\sigma(g-g^*)+b(g-g^*)&=a\sigma g+a\sigma(-g^*)+bg+b(-g^*)\\
&=(a\sigma g+bg)-(a\sigma g^*+bg^*)\\
&=0,
\end{align*}
which implies $g-g^*$ is a polynomial solution of \eqref{eq-3-2}.
By the arbitrariness of $g$ and $g'$, we can get the result.

Moreover, since the isomorphism $\sigma$ keeps the total degree of polynomials unchanged, for the equation \eqref{eq-3-1}, there exist a polynomial solution $g^*\in\mathbb{F}[\alpha,\beta]$ with $\deg g^*=\deg f-\max\{\deg a, \deg b\}$.
\end{proof}

Note that, for the equation \eqref{eq-3-1}, we could get the polynomial solution $g^*$ by the undetermined coefficient method.

Next, we only need consider the equation \eqref{eq-3-2} for the case $\deg a=\deg b$, since $g=0$ when $\deg a\neq\deg b$. Assume $\deg a =\deg b=n$, let $a=\sum_{i=0}^na_i$, $b=\sum_{i=0}^nb_i$ be the homogeneous decomposition of $a$ and $b$, respectively, and suppose $\deg g=d$ and $g=\sum_{i=0}^dg_i$ is the homogeneous decomposition of $g$, then
\begin{align}\label{eq-3-3}
a_n\cdot\sigma(g_d)+b_n\cdot g_d=0.
\end{align}
In order to calculate $d$, we do the following algorithm:
\begin{algorithm}[H]
\renewcommand{\thealgorithm}{}
\caption{}
\small{
\textbf{Step 1:} Let $a_n^0:=a_n/\gcd(a_n,b_n), b_n^0:=b_n/\gcd(a_n,b_n)$ and $\Spr_{\sigma}(a_n^0,b_n^0)=\{k_1<k_2<\ldots<k_N\}$.
\begin{algorithmic}[htb]
  \FOR{$1\leq i \leq N$}
  \STATE{$s_i:=\gcd(a_n^{i-1},\sigma^{k_i}b_n^{i-1});\ a_n^{i}:=\frac{a_n^{i-1}}{s_i};\ b_n^{i}=\frac{b_n^{i-1}}{\sigma^{-k_i}s_i}.$}
\ENDFOR
\STATE{$t(\alpha,\beta):=\prod_{i=1}^N\prod_{j=1}^{k_i}\sigma^{-j}s_i.$}\\
 \end{algorithmic}
 \textbf{Step 2:} Let $\Spr_{\sigma}(b_n^N,a_n^N)=\{l_1<l_2<\ldots<l_M\}$
 \begin{algorithmic}[htb]
 \FOR{$1\leq j \leq M$}
  \STATE{$r_i:=\gcd(b_n^{N+j-1},\sigma^{l_j}a_n^{N+j-1}); b_n^{N+j}:=\frac{b_n^{N+j-1}}{r_j};
  a_n^{N+j}=\frac{a_n^{N+j-1}}{\sigma^{-l_i}r_i}.$}
\ENDFOR
\STATE{$h(\alpha,\beta):=\prod_{i=1}^M\prod_{j=1}^{l_i}\sigma^{-j}rn_i.$\\ }
\end{algorithmic}}
\end{algorithm}
Note that $a_n^0, b_n^0$ are two primely homogenous polynomials, by Theorem \ref{thm-prime}, we have $\Spr_{\sigma}(a_n,b_n)$  and $\Spr_{\sigma}(b_n^N,a_n^N)$ are finite.

By the algorithm, we could find that \eqref{eq-3-2} implies
\[\frac{\sigma(g_d\cdot t/h)}{g_d\cdot t/h}=-\frac{b_{n}^{N+M}}{a_n^{N+M}},
\]
where $\Spr_{\sigma}(a_n^{N+M}, b_n^{N+M})=\Spr_{\sigma}(b_n^{N+M}, a_n^{N+M})=\emptyset$,
and
\[\deg t=k_1\cdot \deg s_1+k_2\cdot \deg s_2+\cdots+k_N\cdot \deg s_N,\quad\deg h=l_1\cdot \deg r_1+l_2\cdot \deg r_2+\cdots+l_N\cdot\deg r_N.\]
Then by Theorem \ref{thm-2-1}, we have
\[
\frac{\sigma(g_d\cdot c/h)}{g_d \cdot c/h}=-\frac{b_n^{N+M}}{a_n^{N+M}}\in\mathbb{F}.
\]

\begin{theorem}\label{thm-3}
Let $(\mathbb{F}(\alpha,\beta),\sigma)$ be the bivariate difference field with $\sigma|_{\mathbb{F}}=id$ and matrix $A =
  \left(
    \begin{array}{cc}
      0 & u \\
      1 & v \\
    \end{array}
  \right)$, and $g(\alpha,\beta)=\sum_{k=0}^dg_k$ be a polynomial solution of \eqref{eq-3-2}.
If $A$ has two different eigenvalues $\lambda_1,\lambda_2\in\mathbb{F}$ and $\lambda_1/\lambda_2$ is not a root of unit, then
\begin{align}\label{eq-3-4}
\deg g\leq\max\left\{d\in\mathbb{N}\mid\lambda_1^i\lambda_2^{d+\deg c-\deg h-i}=-b_n^{N+M}/a_n^{N+M},-\deg h\leq i\leq d+\deg c, i\in \mathbb{Z}\right\},
\end{align}
and
\begin{align}\label{eq-3-5}
\left\{k\mid g_k\neq0\right\}\subseteq\left\{d\in\mathbb{N}\mid\lambda_1^i\lambda_2^{d+\deg c-\deg h-i}=-b_n^{N+M}/a_n^{N+M},-\deg h\leq i\leq d+\deg c, i\in \mathbb{Z}\right\},
\end{align}
where $a_n^{N+M},b_n^{N+M},c,h$ are all constructed by Algorithm.
\end{theorem}
\begin{proof}
Note that it suffices to show \eqref{eq-3-5}. For each $g_k\neq0$ with $0\leq k\leq d$, according to Algorithm, we have
\[
a_n\sigma(g_k)+b_ng_k=0,\quad \text{and}
\quad \frac{\sigma(g_k\cdot c/h)}{g_k\cdot c/h}=-\frac{b_n^{N+M}}{a_n^{N+M}}\in\mathbb{F}.
\]
Since $g_k, c, h$ are homogenous polynomials, by Theorem \ref{thm-2-3}, there exists $i$ with $-\deg h\leq i\leq k+\deg c$ such that
\[
\frac{\sigma(g_k\cdot c/h)}{g_k \cdot c/h}=-\frac{b_n^{N+M}}{a_n^{N+M}}=\lambda_1^i\lambda_2^{k+\deg c-\deg h-i},
\]
which implies \eqref{eq-3-5} holds.
\end{proof}
Note that there maybe does not exist the maximum $d$, i.e., the set in \eqref{eq-3-4} is a infinite set, so we could get a series of polynomial solutions using the undetermined method. Moreover, for each $0\leq i\leq n$, $a_i\sigma g_d+b_i\sigma g_d=0$, and we could solve the polynomial $g_d$ for each $i$ by the above algorithm, which simplifies the calculation.
\begin{example}
For the bivariate difference field $(\mathbb{R}(\alpha,\beta),\sigma)$ involving the Fibonacci numbers $F_n$, consider the difference equation (\ref{eq-3-1}) with $a=1,b=-1$ and $f(\alpha, \beta)=\alpha^3$, i.e., find all the polynomial solutions $g(\alpha, \beta)\in\mathbb{F}(\alpha,\beta)$ such that
\begin{align}\label{eq-3-9}
\sigma\left(g(\alpha, \beta)\right)-g(\alpha, \beta)
=\alpha^2.
\end{align}
\end{example}
\begin{proof}
In the bivariate difference field $(\mathbb{R}(\alpha,\beta),\sigma)$ involving the Fibonacci numbers $F_n$, we could compute the eigenvalues $\lambda_1, \lambda_2$ and the corresponding eigenvectors $X_1, X_2$ of the matrix $A=\left(
    \begin{array}{cc}
      0 & 1 \\
      1 & 1 \\
    \end{array}
  \right)$ are
\[\lambda_1=\frac{1+\sqrt{5}}{2},\quad\lambda_2=\frac{1-\sqrt{5}}{2}, \text{ and  } X_1=\left(1,\frac{1+\sqrt{5}}{2}\right)^T,\ X_2=\left(1,\frac{1-\sqrt{5}}{2}\right)^T.
\]
First, using the undetermined coefficient method with $\deg g^*=\deg f-\max\{\deg a, \deg b\}=3$, we could find a polynomial solution $g^*=\frac{1}{2}\left(-3\alpha^3-\beta^3+3\alpha\beta\right)$ of \eqref{eq-3-9}.

Second, we character all the polynomial solutions $g\in\mathbb{F}[\alpha,\beta]$ of
\begin{align}\label{eq-3-10}
\sigma\left(g(\alpha, \beta)\right)-g(\alpha, \beta)
=0.
\end{align}
Obviously, $\frac{\sigma g}{g}=1$, then by Theorem \ref{thm-3}, we have $-b_n^{N+M}/a_n^{N+M}=1, c=h=0$, and the set
\[\left\{d\in\mathbb{N}\mid\lambda_1^i\lambda_2^{d-i}=1,0\leq i\leq d, i\in \mathbb{Z}\right\}=\{d\mid d=2i, i\in\mathbb{N}\text{ is even}\},\]
hence, by Theorem \ref{thm-2-2}, we could get, for each $0\leq i\leq d$,
\[
g_i\in\left\{c\cdot\left[\left(\alpha+\left(\frac{1+\sqrt{5}}{2}\right)\beta\right)\cdot\left(\alpha+\left(\frac{1-\sqrt{5}}{2}\right)\beta\right)\right]^{i}
\mid c\in\mathbb{F}, i\in\mathbb{N}\text{ is even}\right\},
\]
which implies all the polynomial solutions of \eqref{eq-3-10} and \eqref{eq-3-9}.
\end{proof}
Note that, choose the polynomial solution $g$ of \eqref{eq-3-10} with $\deg g=0$, we can get
\[\sigma\left(\frac{1}{2}\left(-3\alpha^3-\beta^3+3\alpha\beta\right)\right)-\frac{1}{2}\left(-3\alpha^3-\beta^3+3\alpha\beta\right)=\alpha^3,\]
which implies
\[
\sum_{n=1}^mF_{n}^3=\frac{1}{2}\left(-F_{n+2}^3-3F_{n+1}^3+3F_{n+1}F_{n+2}^2+1\right).
\]
Also, since $\frac{1}{2}\left(-3\alpha^3-\beta^3+3\alpha\beta\right)=\frac{1}{4}\left(\beta^3-3\alpha^3+3\left(\beta-\alpha\right)\left(\alpha^2+\alpha\beta-\beta^2\right)\right)$ and $\sigma\left(\alpha^2+\alpha\beta-\beta^2\right)=-\left(\alpha^2+\alpha\beta-\beta^2\right)$, we could get a result of Rao, published by Koshy\citep[Page 89, Eq(38)]{Koshy-2001},
\[
\sum_{n=1}^mF_{n}^3=\frac{1}{4}\left(F_{n+2}^3-3F_{n+1}^3+3(-1)^nF_{n}+2\right).
\]
Choose the polynomial solution $g$ of \eqref{eq-3-10} with $c=1$ and $\deg g=4$, we can get
\[
\sigma\left(\frac{1}{2}\left(-3\alpha^3-\beta^3+3\alpha\beta\right)+\left(\alpha^2+\alpha\beta-\beta^2\right)^{2}\right)
-\left(\frac{1}{2}\left(-3\alpha^3-\beta^3+3\alpha\beta\right)+\left(\alpha^2+\alpha\beta-\beta^2\right)^{2}\right)=\alpha^3,
\]
which implies
\begin{align*}
\sum_{n=1}^mF_{n}^3=F_{m+1}^4+2F_{m+1}^3F_{m+2}-F_{m+1}^2F_{m+2}^2-2F_{m+1}F_{m+2}^3+F_{m+2}^4+\frac{1}{2}\left(-F_{n+2}^3-3F_{n+1}^3+3F_{n+1}F_{n+2}^2+1\right)\\
=F_{m+1}^4+2F_{m+1}^3F_{m+2}-F_{m+1}^2F_{m+2}^2-2F_{m+1}F_{m+2}^3+F_{m+2}^4+\frac{1}{4}\left(F_{n+2}^3-3F_{n+1}^3+3(-1)^nF_{n}+2\right).
\end{align*}
\begin{example}
For the bivariate difference field $(\mathbb{R}(\alpha,\beta),\sigma)$ involving the Pell numbers $P_n$, consider the difference equation (\ref{eq-3-1}) with $a=\alpha^2+\alpha+2\beta, b=\beta^2+2\beta$, and  $f(\alpha, \beta)=\alpha^3+\beta^3+\alpha^2\beta-\alpha\beta^2+\alpha^2+4\beta^2+\alpha\beta$, i.e.,  find all the polynomial solutions $g(\alpha, \beta)\in\mathbb{F}(\alpha,\beta)$ such that
\begin{align}\label{eq-3-11}
(\alpha^2+\alpha+2\beta)\left(g(\alpha, \beta)\right)+(\beta^2+2\beta)g(\alpha, \beta)
=\alpha^3+\beta^3+\alpha^2\beta-\alpha\beta^2+\alpha^2+4\beta^2+\alpha\beta.
\end{align}
\end{example}
\begin{proof}
According to the definition of bivariate difference field, we regard the summations involving $P_n$ lie in the difference field $(\mathbb{R}(\alpha, \beta), \sigma)$ with
\[
\sigma|_{\mathbb{R}} = {\rm id}, \quad  \sigma(\alpha)=\beta,  \quad\text{and}\quad \sigma(\beta)=\alpha+2\beta,
\]
and we could compute the eigenvalues $\lambda_1, \lambda_2$ and the corresponding eigenvectors $X_1, X_2$ of the matrix $A=\left(
    \begin{array}{cc}
      0 & 1 \\
      1 & 2 \\
    \end{array}
  \right)$ are
\[\lambda_1=1+\sqrt{2},\quad\lambda_2=1-\sqrt{2}, \text{ and  } X_1=\left(1,1+\sqrt{2}\right)^T,\ X_2=\left(1,1-\sqrt{2}\right)^T.
\]
First, using the undetermined coefficient method with $\deg g^*=\deg f-\max\{\deg a, \deg b\}=1$, we could find a polynomial solution $g^*=\alpha-\beta$ of \eqref{eq-3-11}.

Second, we character all the polynomial solutions $g\in\mathbb{F}[\alpha,\beta]$ of
\begin{align}\label{eq-3-12}
(\alpha^2+\alpha+2\beta)\sigma\left(g(\alpha, \beta)\right)+(\beta^2+2\beta)g(\alpha, \beta)
=0.
\end{align}
Note that $a_1=\alpha+2\beta, a_2=\alpha^2, b_1=2\beta, b_2=\beta^2$, by Algorithm and Theorem \ref{thm-3}, we could compute $\Spr_{\sigma}(a_1,b_1)=\Spr_{\sigma}(a_2,b_2)=\emptyset$, $\Spr_{\sigma}(b_1,a_1)=\{1\}$, and $\Spr_{\sigma}(b_2,a_2)=\{1\}$, which implies
\[\frac{\sigma(g_d/\beta)}{g_d/\beta}=-2, \quad \frac{\sigma(g_d/\alpha^2)}{g_d/\alpha^2}=-1,\]
then, the set
\[\left\{d\in\mathbb{N}\mid\lambda_1^i\lambda_2^{d-i}=-1,-2\leq i\leq d, i\in \mathbb{Z}\right\}=\{d\mid d=2i+2, i\geq2\text{ is odd}\},\]
and the set
\[\left\{d\in\mathbb{N}\mid\lambda_1^i\lambda_2^{d-i}=-2,-2\leq i\leq d, i\in \mathbb{Z}\right\}=\emptyset,\]
which imply equation \eqref{eq-3-12} has only one polynomial solution $g=0$ and  equation \eqref{eq-3-11} has only one polynomial solution $g=\alpha-\beta$.
\end{proof}

\smallskip{}
\noindent \textbf{Acknowledgement.}
The author would like to thank Qing-Hu Hou for his advice and comments connected with the subject of this paper.




\end{document}